\documentclass[11pt]{amsart}

\usepackage{amssymb}
\usepackage{bm}
\usepackage[centertags]{amsmath}
\usepackage{amsfonts}
\usepackage{amsthm}
\usepackage{mathrsfs}
\usepackage[usenames]{xcolor}
\usepackage[normalem]{ulem}
\theoremstyle{definition}

\newtheorem{thm}{Theorem}[section]
\newtheorem{defn}[thm]{Definition}
\newtheorem{lem}[thm]{Lemma}
\newtheorem{prop}[thm]{Proposition}
\newtheorem{cor}[thm]{Corollary}
\newtheorem{rem}[thm]{Remark}

\newtheorem*{thm*}{Theorem}

\newtheorem*{cor*}{Corollary}

\newtheorem*{prp*}{Proposition}

\newtheorem{problem}{Problem}

\newtheorem*{ntt}{Notation}

\newcommand{\nn}{\mathbb{N}}
\newcommand{\inn}{\in\mathbb{N}}

\newcommand{\ee}{\varepsilon}
\newcommand{\al}{\alpha}

\newcommand{\la}{\lambda}

\newcommand{\vp}{\varepsilon}
\newcommand{\N}{\mathbb{N}}

\newcommand\restr[2]{{
  \left.\kern-\nulldelimiterspace 
  #1 
  \vphantom{\big|} 
  \right|_{#2} 
  }}

\newcommand{\pon}{{p_1^\prime}}
\newcommand{\pze}{{p_{\xi_0}^\prime}}

\DeclareMathOperator{\supp}{supp}

\DeclareMathOperator{\ran}{ran}

\long\def\symbolfootnote[#1]#2{\begingroup%
\def\thefootnote{\fnsymbol{footnote}}\footnote[#1]{#2}\endgroup}

\begin{document}

\title [Disconnected set of $p$'s in Krivine's Theorem]{The stabilized set of $p$'s in Krivine's theorem can be disconnected}
\dedicatory{In memory of Edward Odell}

\author[K. Beanland]{Kevin Beanland}
\address{Department of Mathematics,
    Washington and Lee University,
    Lexington, VA 24450}
\email{beanlandk@wlu.edu}

\author[D. Freeman]{Daniel Freeman}
\address{Department of Mathematics and Computer Science,
Saint Louis University, St Louis, MO 63103} \email{dfreema7@slu.edu}

\author[P. Motakis]{Pavlos Motakis}
\address{National Technical University of Athens, Faculty of Applied Sciences,
Department of Mathematics, Zografou Campus, 157 80, Athens, Greece}
\email{pmotakis@central.ntua.gr}

\maketitle

\symbolfootnote[0]{\textit{2010 Mathematics Subject
Classification:} Primary 46B03, 46B06,  46B07,  46B25, 46B45}

\symbolfootnote[0]{\textit{Key words:} Spreading models, Finite block representability, Krivine's Theorem}
\symbolfootnote[0]{The first named author acknowledges support from the Lenfest Summer Grant at Washington and Lee University.}
\symbolfootnote[0]{Research of the second author was supported by the National Science Foundation.}
\symbolfootnote[0]{The authors would like to acknowledge the support of program API$\Sigma$TEIA-1082.}

\begin{abstract} 
For any  closed subset $F$ of $[1,\infty]$ which is either finite or consists of the elements of an increasing sequence and its limit, a reflexive Banach space $X$ with a 1-unconditional basis is constructed so that in each block subspace $Y$ of $X$, $\ell_p$ is finitely block represented in $Y$  if and only if $p \in F$. In particular, this solves the question as to whether the stabilized Krivine set for a Banach space had to be connected. We also prove that for every infinite dimensional subspace  $Y$ of $X$ there is  a dense subset $G$ of $F$ such that the spreading models admitted by $Y$  are exactly the $\ell_p$ for $p\in G$.  
\end{abstract}

\section{Introduction}

In the past, many of the driving questions in the study of Banach
spaces concerned the existence of ``nice'' subspaces of general
infinite dimensional Banach spaces.  Finding counterexamples to
these questions involved developing new ideas for constructing
Banach spaces. B. Tsirelson's construction of a reflexive
infinite dimensional Banach space which does not contain $\ell_p$
for any $1<p<\infty$ \cite{T} and W.T. Gowers and B. Maurey's construction
of an infinite dimensional Banach space which does not contain an
unconditional basic sequence \cite{GM} are two important examples.  
On the other hand, after Tsirelson's construction, J-L. Krivine 
proved that every basic sequence  contains $\ell_p$ for some $1\leqslant
p\leqslant\infty$ finitely block represented \cite{K} (where the case
$p=\infty$ refers to $c_0$), and  it is not difficult to show that every normalized weakly null sequence in a Banach space has a subsequence with a 1-suppression unconditional spreading model. Thus, though we cannot always find these properties
in infinite dimensional subspaces, they are still always
present in certain finite block or asymptotic structure.

In his paper on Krivine's Theorem, Rosenthal proved that given any
Banach space, the set of $p$'s such that $\ell_p$ is finitely block
represented in the Banach space can be stabilized on a subspace
\cite{R2} (for a simplified proof of the stability result see also \cite[page 133]{M}). That is, given any infinite dimensional Banach space $X$, there
exists an infinite dimensional subspace $Y\subseteq X$ with a basis
and a nonempty closed subset $I\subseteq[1,\infty]$ such that for every block subspace $Z$ of $Y$, $\ell_p$ is
finitely block represented in $Z$ if and only if
$p\in I$.   Rosenthal concluded his
paper by asking if this stabilized Krivine set $I$ had to be a
singleton. E. Odell and Th. Schlumprecht answered this question by
constructing a Banach space $X$ with an unconditional basis which
had the property that every unconditional basic sequence is finitely
block represented in every block sequence in $X$ \cite{OS1}.  Thus,
the stabilized Krivine set for this space is the interval $[1,\infty]$.
Later, Odell and Schlumprecht constructed a Banach space with a
conditional basis which had the property that every  monotone basic sequence
is finitely block represented in every block sequence in $X$
\cite{OS2}.   At this point, the known possible stabilized Krivine
sets for a Banach space are singletons and the entire interval
$[1,\infty]$. P. Habala and N. Tomczak-Jaegermann  proved that if $1\leqslant p< q\leqslant \infty$ and $X$ is an infinite dimensional Banach space such that $\ell_p$ and $\ell_q$ are finitely block represented 
in every block subspace of $X$ then $X$ has a quotient $Z$ so 
that every $r \in[p,q]$ is finitely block represented in $Z$ \cite{HT}.
They then asked if the stabilized Krivine set
for a Banach space is always connected \cite{HT}, which was  later included as problem 12 in
Odell's presentation of 15 open problems in Banach spaces at the
Fields institute in 2002 \cite{O}. We solve the stabilized Krivine set problem with the following theorem.

\begin{thm*}\label{T:main}
Let $F\subseteq [1,\infty]$ be either a finite set or a set consisting of an increasing sequence and its limit. Then there exists a reflexive Banach space $X$  with an unconditional basis such that for every infinite dimensional block subspace $Y$ of $X$:
\begin{itemize}
\item[(i)] For all $1\leqslant p\leqslant\infty$, the space $\ell_p$ is finitely block represented
in $Y$ if and only if $p \in F$.  
\item[(ii)] If $F$ is finite then the spreading models admitted by $Y$ are exactly the spaces $\ell_p$ for $p\in F$.
\item[(iii)] If $F$ is an increasing sequence with limit $p_{\omega}$  then every spreading model admitted by $Y$ is
isomorphic to $\ell_p$ for some $p\in F$ and for every $p\in F\setminus \{p_{\omega}\}$ $\ell_p$ 
is admitted as a spreading model by $Y$.
\end{itemize}

\end{thm*}

This theorem is somewhat surprising in that the corresponding question for finite representability instead of finite block representability is very different.  
Indeed, if $\ell_p$ is finitely representable in a Banach space $X$ for some $1\leqslant p<2$ then $\ell_r$ is finitely representable in $X$ for all $r\in [p,2]$.  However, for $2<p<\infty$ the Banach space $\ell_r$ is finitely representable in $\ell_p$ if and only if $r=2$ or $r=p$. Thus the position in $[1,\infty]$ of the set $F$ of $p$'s that are finitely represented in a space $X$ determines whether $F$ is an interval.
 
Our results show that in the case of block finitely represented, the position of $F$ in the interval $[1,\infty]$ does not matter.

Theorem \ref{T:main} also solves several open questions on spreading models raised by G.
Androulakis, E. Odell, Th. Schlumprecht, and
 Tomczak-Jaegermann \cite{AOST}.  They asked in particular the following three questions: Does there exist a Banach
space so that every subspace has exactly $n$  many different spreading models? 
Does there exist a Banach space so that every subspace has exactly countably infinitely many different spreading models? If a Banach space admits
$\ell_1$ and $\ell_2$ spreading models in every 
subspace must it  admit uncountably many spreading models?
In \cite{AM2}, S.A. Arygros and the third author have constructed a space so that every subspace admits every 
unconditional basis as a spreading model. In
\cite{ABM}, S.A. Argyros with the first and third named authors created
a Banach space such that every infinite dimensional subspace admits
exactly two spreading models up to isomorphism, namely $\ell_1$ and
$c_0$.  Theorem \ref{T:main} includes the case that $F$ is an increasing sequence and its limit, and so it is natural to question if the case of a decreasing sequence is possible.  However, B. Sari proved that if a Banach space admits a countable collection of spreading models which
form a strictly increasing sequence in terms of domination, then the Banach space admits uncountably many spreading models \cite{S}.  Thus, Theorem \ref{T:main} (iii) would be impossible in the case that $F$ is a decreasing sequence and its limit as the spaces $\{\ell_p\}_{p\in F}$ would include an increasing sequence in terms of domination.


Given a Banach space $X$ with a basis, one may consider the set of all $p$'s, such that $\ell_p$ is admitted as a spreading model by $X$. Although this set may fail to coincide with the Krivine set of the space, or may even be empty \cite{OS1}, it is always contained in the Krivine set. Therefore, for a given subset $F$ of $[1,\infty]$ when constructing a space $X$, one way to ensure that $F$ is contained in any stabilized Krivine set of $X$, is to have $\ell_p$ admitted as a spreading model by every subspace of the space for every $p\in F$. For any single $1\leqslant p<\infty$, Tsirelson's method allows one to build a reflexive space not containing $\ell_p$ that is asymptotic $\ell_p$. Every spreading model admitted by this space is isomorphic to $\ell_p$ and the Krivine set of every infinite dimensional subspace is the singleton $\{p\}$. In this paper, for any finite set of $p$'s, we build a space with exactly these $\ell_p$'s hereditarily as spreading models and exactly these $p$'s hereditary as Krivine $p$'s. Moreover, for any increasing sequence of $p$'s we get almost the same result. In this case, the only caveat is that,  although the basis of the space $X$ admits only the limit $p$ as a spreading model, we did not prove that the limit $p$ is admitted in every subspace.

In the case of two distinct $p$'s,  our construction is rooted in the convexified Tsirelson's spaces in the sense of T. Figiel and W.B. Johnson's description
\cite{FJ} and the work of Odell and Schlumprecht \cite{OS1}, \cite{OS2}. The methods we follow are based on the ones from \cite{ABM}. In particular, for the simplest case
 of $F=\{1,\infty\}$, our construction reduces to a small modification of the space $\mathfrak{X}^1_{_{0,1}}$, which is the simplest case of the construction defined in that paper, and in which it is shown that $\mathfrak{X}^1_{_{0,1}}$ admits only $c_0$ and $\ell_1$ spreading models in every subspace. 
In recent literature, the spaces in \cite{OS1}, \cite{OS2}, \cite{ABM}, \cite{AM1} and \cite{AM2} are referred to as Tsirelson spaces with constraints or multi-layer Tsirelson spaces.  For the sake of understanding our construction in the simplest case $F=\{1,\infty\}$,
the norm satisfies the following implicit equation for $x \in c_{00}$:
\
$$\| x  \| =  \| x\|_0 \vee \sup  \frac{1}{4}\sum_{i=1}^n \| E_i x\|_{m_i},$$
where the supremum is over successive intervals $(E_i)_{i=1}^n$ and $(m_i)_{i=1}^n$  with 
$$  (\min E_i)_{i=1}^n \in \mathcal{S}_{1}, \min E_{i} > (\max E_{i-1})^2  \mbox{ and } m_i > \max E_{i-1},$$
and for each $m\inn$,
$$\|x\|_m = \frac{1}{4m} \sup \sum_{i=1}^m \| F_i x \|,$$
where the supremum is over successive intervals $(F_i)_{i=1}^m$. These $m$-norms and the way they are combined above are the previously mentioned contraints.   Since the constraints are based on averages, local and asymptotic $c_0$ structure occurs in every subsapce. Furthermore, in contrast to Tsirelson space, which has homogeneous asymptotic $\ell_1$ structure, the above construction hereditarily provides  both $\ell_1$ and $c_0$ local and asymptotic structure.

In the case that $F = \{p_1<\cdots<p_n\}\subset[1,\infty]$ we present a space $X$, admitting hereditarily $\ell_{p_1},\ldots,\ell_{p_n}$ asymptotic structure and nothing more. For this purpose we define a new norm, which has
 $n$-many layers, each one corresponding to an $\ell_{p_k}$ structure, for $k=1,\ldots,n$. The base layer corresponds to the $\ell_{p_n}$ norm, while for $k=1,\ldots,n-1$ the  $k^{th}$ layer corresponds to norm $\ell_{p_k}$ and it is defined using the previous layers. To avoid the domination of some of these layers over the rest,  to each of these layers, except for the basic one, some constraints have to be applied. The constraints are based on $p_n^\prime$-averages, where $p_n'$ is the conjugate exponent of $p_n$.

When $F$ consists of an increasing sequence $(p_k)_k$ and its limit $p_\omega$, countably many layers of the norm are used. In this case, the norm $\|\cdot\|_*$ of the space is defined through the following formulas.{ We state here the implicit equations for the norms for the sake of giving insight into our construction, but we will actually use a different definition in Section \ref{section def} in terms of norming functionals}. For $0< \theta \leqslant 1/4$ and $x\in c_{00}(\mathbb{N})$ we define:
\begin{equation*}
\|x\|_\omega = \theta\sup\left(\sum_{q=1}^d\|E_qx\|_*^{p_\omega}\right)^{1/p_\omega}\;\text{and}\quad
\|x\|_{0,m} = \theta\sup\frac{1}{m^{1/p_\omega^\prime}}\sum_{q=1}^m\|E_qx\|_*
\end{equation*}
where both suprema are are taken over all $d \inn$ and  successive intervals $(E_q)_{q=1}^d$ of the natural numbers. These $\|\cdot\|_{0,m}$ norms are the constraints applied to the norm of the space.
If for some $k> 0$ the norms $\|\cdot\|_{i,m}$ have been defined for every $0 \leqslant i < k$ and $m\inn$, for $x\in c_{00}(\mathbb{N})$ and $m\inn$ we define:
\begin{equation*}
\|x\|_{k,m} = \theta \sup \left(\sum_{q=1}^d\|E_qx\|_{i_q,m_q}^{p_{k}}\right)^{1/p_{k}}
\end{equation*}
where the supremum is taken over all $d\in\mathbb{N}$, $0 \leqslant i_q <k$ for $q = 1,\ldots,d$ and $(E_q)_{q=1}^d$, $(m_q)_{q=1}^d$ which satisfy certain growth conditions depending on $m$. The norm of the space satisfies the following implicit equation:
\begin{equation*}
\|x\|_* = \max\left\{\|x\|_\infty,\;\|x\|_\omega,\sup\{\|x\|_{k,m}:\;k,m\in\mathbb{N}\}\right\}.
\end{equation*}
Using the above description of the norm, it is easy to see that any block  sequence in our space satisfies a lower $\ell_{p_\omega}$ estimate with constant $\theta$. Likewise, in section \ref{section wiggles} we prove that any block sequence satisfies an upper $\ell_{p_1}$ estimate with constant $2$. In the case $F$ is finite with $F = \{p_1<\cdots<p_n\}$ then the norm of the space satisfies the same formula, where $p_\omega$ is replaced with $p_n$.

The paper is organized as follows. In section 2 we give a few preliminary definitions. Section 3 contains the definition of the spaces. In section 4 we set notation that we will use in our subsequent evaluations and prove upper and lower estimates on normalized block sequences.  In sections 5 and 6 we prove the spaces have the desired spreading model structure. Finally, in section 7 we show that in every block subspace the only Krivine $p$'s are those admitted as spreading models.

The majority of the research included in this paper was conducted while the first two authors where visiting the National Technical University of Athens.  We sincerely thank Spiros  Argyros for his hospitality and enlightening conversations.

\section{ Preliminaries} 
We begin with some preliminary definitions. Two basic sequences $(x_i)$ and $(y_i)$ are {\em$C$-equivalent} for
some $C\geqslant 1$ if $\sqrt{C}^{-1}\|\sum a_i x_i\|\leqslant\|\sum a_i
y_i\|\leqslant \sqrt{C}\|\sum a_i x_i\| $ for all scalar sequences
$(a_i)$. A basic sequence $(e_i)_{i=1}^\infty$ is {\em finitely
block represented} in a basic sequence $(x_i)_{i=1}^\infty$ if for
all $N\in\N$ and $\vp>0$ there exists a finite block sequence
$(y_i)_{i=1}^N$ of $(x_i)_{i=1}^\infty$ which is
$(1+\vp)$-equivalent to $(e_i)_{i=1}^N$.   For $1\leqslant p\leqslant\infty$,
we say that $\ell_p$ is {\em finitely block represented} in
$(x_i)_{i=1}^\infty$ if the unit vector basis of $\ell_p$ is
finitely block represented in $(x_i)_{i=1}^\infty$ (where we use the
case $p=\infty$ to mean $c_0$).

We say that a basic sequence $(e_i)_{i=1}^\infty$ is a {\em
spreading model} of a basic sequence $(x_i)_{i=1}^\infty$ if for all
finite sequences of scalars $(a_i)_{i=1}^n$ we have that
$$\|\sum_{i=1}^n a_i
e_i\|=\lim_{t_1\rightarrow\infty}\cdots\lim_{t_n\rightarrow\infty}\|\sum_{i=1}^n
a_i x_{t_i}\|.
$$
We say that a Banach space $X$ {\em admits} $(e_i)_{i=1}^\infty$ as
a spreading model if $(e_i)_{i=1}^\infty$  is equivalent to a spreading model of
some basic sequence in $X$. We say that $X$ admits $\ell_p$
as a spreading model if $X$ admits a spreading model
equivalent to the unit vector basis for $\ell_p$.  In the
literature, the basic sequence $(e_i)_{i=1}^\infty$ as well as the
Banach space formed by its closed span are both often referred to as
spreading models.

\section{The definition of the space $X$.}\label{section def}

In this section we give the definition of the norming set of the space $X$. We apply a variation of the method of saturation under constraints, introduced by Odell and Schlumprecht in \cite{OS1}, \cite{OS2}. The way this method is applied is similar to the one in \cite{ABM} and it allows $\ell_p$ structure to appear hereditarily in the space, for a predetermined set of $p$'s, which is either finite or consists of an increasing sequence and its limit.

\begin{ntt}
Let $G$ be a subset of $c_{00}(\mathbb{N})$.
\begin{enumerate}

\item[(i)] A finite sequence $(f_q)_{q=1}^d$ of elements of $G$
 will be called {\em admissible} if $f_1 <\cdots < f_d$
and $d \leqslant \min\supp f_1$.

\item[(ii)] Assume that 
$(f_q)_q$ is a sequence of functionals in $G$, and each functional $f_q$ has been assigned a positive integer $s(f_q)$, called the size of $f_q$.  Then $(f_q)_q$ will be called {\em very fast growing } if $(\max\supp f_{q-1})^2 < \min\supp f_q$ and $s(f_q) > \max\supp f_{q-1}$ for all $q>1$.

\end{enumerate}
\end{ntt}



Let $\xi_0 \in [2,\omega]$ and let $F=\{p_k :1 \leqslant k <\xi_0\} \cup \{p_{\xi_0}\} \subset [1, \infty]$ with $p_k \uparrow p_{\xi_0}$ in the case that $\xi_0=\omega$ and $p_1< p_2 < \cdots <p_{\xi_0-1} < p_{\xi_0}$ otherwise.

We now define the norming set of the space $X$. We do so inductively by defining an increasing sequence of subsets of $c_{00}(\mathbb{N})$. To some of the functionals that we construct we shall assign an order, a size or both. Fix a positive real number $0<\theta\leqslant 1/4$. Let $W_0 = \{ \pm e_j^* \}_{j \inn}$. To the functionals in $W_0$ we don't assign an order or size.  Assume that for some $m\inn\cup\{0\}$ the set $W_m$ has been defined, below we describe how the set $W_{m+1}$ is defined.

\subsection*{Functionals of order-$0$, or of order-$\xi_0$.}
Define
$$W^0_{m+1} = \left\{ \theta\sum_{q=1}^d c_qf_q : ~f_1 < \cdots < f_d \in W_m,\; \sum_{q=1}^d|c_q|^{p_{\xi_0}^\prime}\leqslant 1\right\}.$$

A functional $f = \theta\sum_{q=1}^d c_qf_q$ as above will be called of order-$0$. In some cases, for convenience these functionals shall also be referred to as functionals of order-$\xi_0$.

If a functional $f$ of order-$0$ has the form $f = \theta\sum_{q=1}^d (1/n)^{1/{p_{\xi_0}^\prime}}f_q$ with $d\leqslant {n}$, then the size of $f$ is defined to be $s(f) = {n}$. If a functional $f$ of order-$0$ is not of this form then we do not assign any size to it.

If $m + 1=1$ then we define $W_{1} = W_0\cup W_1^0$ , otherwise $m+1 \geqslant 2$ and we shall include more functionals in $W_{m+1}$, as described below.

\subsection*{Functionals of order-$k$, with $1\leqslant k <\xi_0$.}
Define
\begin{equation*}
\begin{split}
W_{m+1}^k = \Bigg\{& \theta \sum_{q=1}^d c_q f_q: ~\sum_{q=1}^d |c_q|^{p_{k}^\prime}\leqslant 1, ~(f_q)_{q=1}^d \mbox{ is an admissible and very}\\
& \mbox{fast growing sequence of functionals in $W_m$, each one}\\
&\mbox{of which has order strictly smaller than $k$}\Bigg\}.
\end{split}
\end{equation*}
A functional $f = \theta\sum_{q=1}^d c_qf_q$ as above will be called of order-$k$ with size $s(f) = \min\{s(f_q):q=1,\ldots,d\}$.

If $k = 1$ and $p_1 = 1$, we replace the condition $\sum_{q=1}^d |c_q|^\pon\leqslant 1$ with the condition $\max\{|c_q|:\;q=1,\ldots,d\}\leqslant 1$.
Note that if $\xi_0$ is finite and $\xi_0 = k_0 + 1$, then very fast growing sequences of functionals of order-$k_0$ are not used. Observe also that some functionals may be of more than one order or have multiple sizes, however, this shall not cause any problems.

If $m+1\geqslant 2$, let $W_{m+1} = \left(\cup_{0\leqslant k<\xi_0} W_{m+1}^k\right)\cup W_m$ and $W= \cup_{m=0}^\infty W_m.$ The space $X$ is the completion of $c_{00}(\mathbb{N})$ under the norm induced by $W$, i.e. for $x\in c_{00}(\mathbb{N})$ the norm of $x$ is equal to $\sup\{|f(x)|:\;f\in W\}$.

\begin{rem}
The following are clear from the definition of the norming set.
\begin{enumerate}

\item[(i)] For every $f_1<\cdots<f_d$ in $W$ and real numbers $(c_q)_{q=1}^d$ with $\sum_{q=1}^d|c_{q_0}|^{\pze}\leqslant 1$, the functional $f = \theta\sum_{q=1}^dc_qf_q$ is also in $W$.

\item[(ii)] For every $1\leqslant k < \xi_0$ and every admissible and very fast growing $f_1<\cdots<f_d$ in $W$, each one of which has order strictly smaller than $k$ and real numbers $(c_q)_{q=1}^d$ with $\sum_{q=1}^d|c_{q_0}|^{p_k^\prime}\leqslant 1$, the functional $f = \theta\sum_{q=1}^dc_qf_q$ is also in $W$.

\end{enumerate}
\end{rem}

\begin{rem}
For every $f\in W$ and  subset $E$ of the natural numbers, we have that $f|_E$, the restriction of $f$ onto $E$, is also in $W$. In particular, if $f$ is of order-$k$, then $f|_E$ is also of order-$k$ and $s(f|_E) \geqslant s(f)$. One can also check that the norming set $W$ is closed under changing signs, i.e. if $f\in W$ and $g$ is such that $|f| = |g|$, then $g$ is also in $W$. Therefore, the unit vector basis of $c_{00}(\mathbb{N})$ forms a $1$-unconditional basis for $X$.
\end{rem}

Recall that functionals of order-$0$ are also called functionals of order-$\xi_0$.
\begin{rem}\label{useful for inductive proofs on the norming set and other such stuff}
For every $m>0$, $1\leqslant \zeta \leqslant \xi_0$ and $f\in W_m$, which is of order-$\zeta$, there exist $f_1 < \cdots < f_d$ in $W_{m-1}$ and real numbers $c_1,\ldots,c_d$ with $\sum_{q=1}^d|c_q|^{p_\zeta^\prime}\leqslant 1$ such that $f = \theta\sum_{q=1}^dc_q f_q$. If moreover $\zeta = k < \xi_0$, then $(f_q)_{q=1}^d$ is an admissible and very fast growing sequence of functionals, each one of which has order strictly smaller than $k$.
\end{rem}

Before proceeding to the study of the properties of the space $X$, let us briefly explain the ingredients of the norming set $W$, without getting into too many details. If $1<\xi_0\leqslant \omega$ and we have determined a set $F = \{p_1<\cdots<p_{\xi_0}\}\subset[1,\infty]$, then every element $f$ of the norming set falls into one of the following three categories:
\begin{itemize}

\item[(i)] The functional $f$ is an element of the basis, i.e. $f\in\{\pm e_i\}_i$.

\item[(ii)] The functional $f$ is of order-$0$, i.e.
$ f = \theta\sum_{q=1}^d c_q f_q $
where $f_1<\cdots<f_d$ can be any successive elements of the norming set, combined with coefficients $(c_q)_q$ in the unit ball of $\ell_{p_{\xi_0}'}$.

\item[(iii)] The functional $f$ is of order-$k$ with $1\leqslant k < \xi_0$, i.e.
$ f = \theta\sum_{q=1}^d c_q f_q $
where the sequence $f_1<\cdots<f_d$ are successive elements of the norming set satisfying certain constraints, while the coefficients $(c_q)_q$ are in the unit ball of $\ell_{p_k'}$.
\end{itemize}

The functionals of order-$0$ provide $\ell_{p_{\xi_0}}$ structure to the space and, since the $\ell_{p_{\xi_0}}$ is the smallest of the $\ell_{p}$ norms for $p\in F$, their construction  is not subject to any constraints. On the other hand, for $1\leqslant k < \xi_0$, the functionals of order $k$ provide $\ell_{p_k}$ structure. One has to define these functionals carefully, in order not to demolish the desired $\ell_{p_\zeta}$ structure, for $k<\zeta\leqslant\xi_0$. This is the role of the constraints, which become more restrictive as $k$ becomes smaller.

One can verify that the norm induced by the norming set $W$ is alternatively described by the implicit formula given in the introduction.





\section{Basic norm evaluations on block sequences of $X$.}\label{section wiggles}

In this section we prove a simple, but useful, lemma and we also prove that block sequences in $X$ have an upper $\ell_{p_1}$ estimate and a lower $\ell_{p_{\xi_0}}$ estimate. We start with some notation, which in conjunction with the next lemma, will be used frequently throughout the paper.  Here, the range of a vector is the smallest closed interval containing the support.

\begin{ntt}
Let $x_1 < \cdots < x_m$ be a finite block sequence in $X$ and $f = \theta\sum_{q=1}^d c_qf_q$ be a functional of order-$\zeta$, $1\leqslant \zeta \leqslant \xi_0$. Define the following :
\begin{eqnarray*}
A_1 &=& \left\{q\in\{1,\ldots,d\}: \ran f_q\cap\ran x_j\neq\varnothing\;\text{for at most one}\; 1\leqslant j\leqslant m\right\},\\
A_2 &=& \{1,\ldots,d\}\setminus A_1,\\
B &=& \{j\in\{1,\ldots,m\}:\;\text{there exists}\;q\in A_1\;\text{with}\;\ran f_q\cap\ran x_j\neq\varnothing\}\\
A_1^j &=& \{q\in A_1: \ran f_q\cap\ran x_j\neq\varnothing\}\;\text{for}\; j\in B,\\
C_j &=& \left\|(c_q)_{q\in A_1^j}\right\|_{\ell_{p_\zeta^\prime}}\;\text{for}\; j\in B,\\
g_j &=& \theta\sum_{q\in A_1^j}(c_q/C_j)f_q\;\;\text{for}\; j\in B\;\text{and}\\
E_q &=& \{j\in\{1,\ldots,m\}:\;\ran f_q\cap \ran x_j\neq\varnothing\},\;\text{for}\;q\in A_2.
\end{eqnarray*}
\end{ntt}

The following lemma follows immediately from our choice of notation.  As in Remark \ref{useful for inductive proofs on the norming set and other such stuff}, here we also use the fact that order-$0$ functionals can be referred to as order-$\xi_0$ functionals.

\begin{lem}
Let $x_1 < \ldots < x_m$ be a finite block sequence in $X$ and $f = \theta\sum_{q=1}^dc_qf_q$ be a functional of order-$\zeta$ for some $1\leqslant \zeta \leqslant \xi_0$.
The functionals $(g_j)_{j\in B}$ are order-$\zeta$ functionals in $W$, we have that $(\sum_{j\in B}C_j^{{p_\zeta^\prime}})^{1/{p_\zeta^\prime}}\leqslant 1$, and the following holds:
\begin{equation}\label{E1}
\left|f\left(\sum_{j=1}^mx_j\right)\right| \leqslant \left(\sum_{j\in B}C_j|g_j(x_j)|\right) + \theta\left(\sum_{q\in A_2}|c_q|\cdot\left|f_q\left(\sum_{j\in E_q}x_j\right)\right|\right).
\end{equation}

Moreover, if $A_2 = \{q_1<\cdots<q_r\}$ then $\max E_{q_i} \leqslant \min E_{q_{i+1}}$ for $i=1,\ldots,r-1$ and  $\max E_{q_i} < \min E_{q_{i+2}}$ for $i=1,\ldots,r-2$.\label{action analysis}\label{some details} Thus, for each $1\leqslant j\leqslant m$ there exists at most two sets $E_q$ such that $x_j\in E_q$.
\end{lem}

Note that  applying H\"older's inequality to \eqref{E1}, gives the following inequality, which in most cases will be more convenient for us than \eqref{E1}.
\begin{equation}\label{E2}
\left|f\left(\sum_{j=1}^m x_j\right)\right| \leqslant \left\|\left(g_j(x_j)\right)_{j\in B}\right\|_{\ell_{p_{\zeta}}} + \theta\Big\|\Big(f_q(\sum_{j\in E_q}x_j)\Big)_{q\in A_2}\Big\|_{\ell_{p_{\zeta}}}.
\end{equation}





\begin{prop}\label{general estimate}
Let $x_1 < \cdots < x_m$ be a normalized finite block sequence in $X$ and $(\lambda_j)_{j=1}^m$ be scalars. The following holds:
\begin{equation*}
\theta \|(\lambda_j)_j\|_{\ell_{p_{\xi_0}}} \leqslant \bigg\|\sum_{j=1}^m \lambda_j x_j\bigg\| \leqslant 2 \|(\lambda_j)_j\|_{\ell_{p_1}}.
\end{equation*}
\end{prop}

\begin{proof}
We first prove the lower inequality. Note that this is trivial in the case that $p_{\xi_0}=\infty$
, thus we assume that $p_{\xi_0}<\infty$. For each $j \in \{1, \ldots, m\}$ find $f_j$ so that $f_j(x_j)= 1$ and $\supp f_j = \supp x_j$.  Without loss of generality, we may assume that 
$(\sum_{i=1}^m |\lambda_i|^{p_{\xi_0}})^{1/p_{\xi_0}}=1$ and $\lambda_i\geq0$ for all $1\leqslant i\leqslant m$.  Thus, $\theta \sum_{j=1}^m |\lambda_j|^{p_{\xi_0}/p^\prime_{\xi_0}} f_j \in W$. Therefore
$$\bigg\|\sum_{j=1}^m \lambda_jx_j\bigg\| \geqslant \theta \sum_{j=1}^m |\lambda_j|^{p_{\xi_0}/p^\prime_{\xi_0}} f_j \bigg(\sum_{i=1}^m \lambda_i x_i\bigg) = \theta\bigg(\sum_{i=1}^m |\lambda_i|^{p^\prime_{\xi_0}}\bigg)=\theta.$$
The upper inequality clearly follows from the following claim that we will prove by induction on $n\in\N$.  For all  $n \in \nn \cup \{0\}$ and $f \in W_n$  (see Remark \ref{useful for inductive proofs on the norming set and other such stuff}) we have 
\begin{equation}\label{IH}
f\bigg(\sum_{j=1}^m \lambda_j x_j\bigg) \leqslant 2\left(\sum_{j=1}^m|\lambda_j|^{p_1}\right)^{1/p_1}.
\end{equation}
The case of $f \in W_0=\{\pm e^*_j\}$ is trivial. Assume that the above holds for some $n \geqslant 0$. Let $f \in W_{n+1}$. Then $f=\theta \sum_{q=1}^d c_q f_q$ is of order-$\zeta$ and $f_1< \cdots < f_d$ are in in $W_{n}$ and $\sum_{q=1}^d |c_q|^{p_\zeta'} \leqslant 1$.

By \eqref{E2} after Lemma \ref{action analysis} then applying the inductive hypothesis, we obtain the following: 
\begin{eqnarray}
\left|f(\sum_{j=1}^m \lambda_j x_{j})\right| &\leqslant& \left(\sum_{j\in B}|\lambda_j|^{p_\zeta}\right)^{1/p_\zeta}\!\!\!\!\!\!\!\!\!\! + \theta\left(\sum_{q\in A_2}\left|f_q\left(\sum_{j\in E_q} \lambda_j x_{j}\right)\right|^{p_\zeta}\right)^{1/p_\zeta}\nonumber\\
&\leqslant& \left(\sum_{j\in B}|\lambda_j|^{p_\zeta}\right)^{1/p_\zeta} + 2\theta\left(\sum_{q\in A_2}\left(\sum_{j\in E_q}|\lambda_j|^{p_{1}}\right)^{p_\zeta/p_{1}}\right)^{1/p_\zeta}.\label{general estimate 1}
\end{eqnarray}
 By the last part of Lemma \ref{some details}, for each $j$ there exists at most two distinct $q\in A_2$ such that $j\in E_q$.  This fact together with $p_{1} \leqslant p_\zeta$ imply that
\begin{equation}
\left(\sum_{q\in A_2}\left(\sum_{j\in E_q}|\lambda_j|^{p_{1}}\right)^{p_\zeta/p_{1}}\right)^{1/p_\zeta} < 2^{1/p_{1}}\left(\sum_{j=1}^m|\lambda_j|^{p_{1}}\right)^{1/p_{1}}.\label{general estimate 2}
\end{equation}
Combining relations \eqref{general estimate 1} and \eqref{general estimate 2} together with $0<\theta\leqslant 1/4$, we obtain the desired bound in \eqref{IH}.


\end{proof}

\section{Spreading models of $X$.}

In this section we define the $\al$-indices in a very similar manner as they have been defined in \cite{ABM}, \cite{AM1} and \cite{AM2}. Although previously the $\al$-indices were used to describe the action of certain averages of functionals on a block sequence, in our case this is not exactly the same.  Here, the indices are used to study the action of functionals of a certain order on a block sequence. However, the principle is the same and we retain this notation. As is the case in these papers, the indices  determine the spreading models admitted by a block sequence in the space $X$. As a consequence we prove that every spreading model admitted by a weakly null sequence in $X$ must equivalent to the unit vector basis of $\ell_{p_\zeta}$ for some $\zeta\in[1,\xi_0]$.

\begin{defn}
Let $(x_j)_j$ be a block sequence in $X$ and let $1\leqslant k < \xi_0$. Assume that for every very fast growing sequence $(f_q)_q$ of functionals in $W$, each one of which has order strictly smaller than $k$, and every subsequence $(x_{j_i})_i$ of $(x_j)_j$ we have that $\lim_i|f_i(x_{j_i})| = 0$. Then we say that the $\al_{<k}$-index of $(x_j)_j$ is zero and write $\al_{<k}\{(x_j)_j\} = 0$. Otherwise we write $\al_{<k}\{(x_j)_j\} > 0$.
\end{defn}

\begin{rem}
Let $(x_j)_j$ be a block sequence in $X$ and $1\leqslant m < k < \xi_0$. If $\al_{<k}\{(x_j)_j\} = 0$ then also $\al_{<m}\{(x_j)_j\} = 0$.
\end{rem}

The following characterization has appeared in similar forms in \cite{ABM}, \cite{AM1} and \cite{AM2}. We omit the proof as it is simple and straightforward.

\begin{prop}
Let $1\leqslant k < \xi_0$ and $(x_j)_j$ be a block sequence in $X$. The following assertions are equivalent:
\begin{enumerate}

\item[(i)] $\al_{<k}\{(x_j)_j\} = 0$.

\item[(ii)] For every $\ee>0$ there exist $j_0, i_0\inn$ such that for every $f\in W$ of order strictly smaller than $k$ with $s(f) \geqslant i_0$ and every $j\geqslant j_0$ we have that $|f(x_j)| < \ee$.

\end{enumerate}
\label{index char}
\end{prop}

\begin{lem}
Let $1\leqslant k < \xi_0$ and $(x_j)_j$ be a bounded  block sequence in $X$ such that $\al_{<k}\{(x_j)_j\} > 0$.  Then $(x_j)_j$ has a subsequence with a  spreading model  that dominates
the unit vector basis for $\ell_{p_k}$.  That is, there exists $\ee>0$ and  a subsequence $(x_{j_i})_i$ of $(x_j)_j$ such that for every natural numbers $m\leqslant i_1 < \cdots < i_m$ and every real numbers $\lambda_1,\ldots,\lambda_m$ the following holds:
\begin{equation*}
\left\|\sum_{t=1}^m\lambda_t x_{j_{i_t}}\right\| \geqslant \ee \left\|(\lambda_t)_t\right\|_{\ell_{p_k}}.
\end{equation*}\label{lower estimate}
\end{lem}

\begin{proof}
By the definition of the $\al_{<k}$ index, there exists $\ee^\prime > 0$, a subsequence of $(x_j)_j$, again denoted by $(x_j)_j$ and a very fast growing sequence $(f_j)_j$ of functionals of order strictly smaller than $k$, such that $|f_j(x_j)| > \ee^\prime$ for all $j\inn$. We may also assume that $\ran f_j \subset \ran x_j$ for all $j\inn$.
Set $\ee = \theta\ee^\prime$ and note that for every $m \leqslant j_1 < \cdots <j_m$ and every real numbers $(c_t)_{t=1}^m$ with $\sum_{t=1}^m|c_t|^{p_k^\prime}\leqslant 1$ the functional $f = \theta\sum_{t=1}^mc_tf_{j_t}$ is of order-$k$. Let $m\leqslant j_1 < \cdots <j_m$ be natural numbers and $\lambda_1,\ldots,\lambda_m$ be real numbers. We have the following estimate:
\begin{eqnarray*}
\left\|\sum_{t=1}^m\lambda_t x_{j_t}\right\| &\geqslant& \sup\left\{\theta\sum_{t=1}^mc_tf_{j_t}\left(\sum_{t=1}^m\lambda_t x_{j_t}\right):\;\sum_{t=1}^m|c_t|^{p_k^\prime}\leqslant 1\right\}\\
&=&  \sup\left\{\theta\sum_{t=1}^m|c_t\lambda_t|\cdot|f_{j_t}(x_{j_t})|:\;\sum_{t=1}^m|c_t|^{p_k^\prime}\leqslant 1\right\}\\
&\geqslant & \theta\ee^\prime \sup\left\{\sum_{t=1}^m|c_t\lambda_t|:\;\sum_{t=1}^m|c_t|^{p_k^\prime}\leqslant 1\right\} \\
&=& \ee\left(\sum_{t=1}^m|\lambda_t|^{p_k}\right)^{1/p_k}.
\end{eqnarray*}
\end{proof}

\begin{lem}
Let $(x_j)_j$ be a normalized block sequence and $2\leqslant \zeta \leqslant \xi_0$  such that $\al_{<k}\{(x_j)_j\} = 0$ for every $1\leqslant k < \zeta$. 
 Then $(x_j)_j$ has a subsequence with a spreading model that is 2-dominated by the unit vector basis for $\ell_{p_\zeta}$. 
In particular, there exists  a subsequence $(x_{j_i})_i$ of $(x_j)_j$ such that for every natural numbers $m\leqslant i_1 < \cdots < i_m$ and every real numbers $\lambda_1,\ldots,\lambda_m$ the following holds:
\begin{equation*}
\left\|\sum_{t=1}^m\lambda_t x_{j_{i_t}}\right\| \leqslant 3\left\|(\lambda_t)_t\right\|_{\ell_{p_\zeta}}.
\end{equation*}\label{upper estimate}
\end{lem}

\begin{proof}
We first consider the case in which $\zeta$ is finite, i.e. $\zeta = k^\prime + 1$ with $1\leqslant k^\prime <\xi_0$.


Using Proposition \ref{index char} we pass to a subsequence, again denoted by $(x_j)_j$ such that for any
$j \geqslant j_0\geqslant 2$ and any
$f\in W$ of order strictly smaller than $k^\prime$ with
$s(f) \geqslant \min\supp x_{j_0}$ we
have that
\begin{equation}
|f(x_j)| < \left(j_0\max\supp x_{j_0-1}\right)^{-1} .\label{constraint action}
\end{equation}
We will show by induction on $n$, where $W = \cup_nW_n$ (see Remark \ref{useful for inductive proofs on the norming set and other such stuff}) that for every $m\leqslant j_1 < \cdots < j_m$, every real numbers $\lambda_1,\ldots,\lambda_m$ and every $f\in W_n$ the following holds:

\begin{equation}\label{E3}
\left|f\bigg(\sum_{t=1}^m\lambda_t x_{j_{t}}\bigg)\right| \leqslant 2\left\|\left(\lambda_t\right)_{t=1}^m\right\|_{\ell_{p_{\zeta}}}.
\end{equation}

For $f\in W_0$ the result holds. Let $f = \theta\sum_{q=1}^dc_qf_q$ be a functional in $W_n$. We distinguish two cases, concerning  the order of $f$.

\noindent {\em Case 1:} The functional $f$ is of order-$\eta$ with $\zeta \leqslant \eta \leqslant \xi_0$.  By Inequality \eqref{E2} after Lemma \ref{action analysis}, we have that


\begin{eqnarray}
\left| f\left(\sum_{t=1}^m  \lambda_tx_{j_{t}}\right)\right| \leqslant \left(\sum_{t\in B}|\lambda_t|^{p_\eta}\right)^{1/p_\eta} + \theta\left(\sum_{q\in A_2}\left|f_q\left(\sum_{t\in E_q}\lambda_t x_{j_t}\right)\right|^{p_\eta}\right)^{1/p_\eta}\nonumber\\
\leqslant \left(\sum_{t\in B}|\lambda_t|^{p_\eta}\right)^{1/p_\eta}+ \theta 2\left(\sum_{q\in A_2}\left(\sum_{t\in E_q}|\la_t|^{p_{k}}\right)^{p_\eta/p_{k}}\right)^{1/p_\eta}\label{upper estimate 1} \textrm{by }\eqref{E3}.
\end{eqnarray}

The fact that $p_{\zeta} \leqslant p_\eta$ and  for each $1\leqslant t\leqslant m$ there exists at most two distinct $q\in A_2$ such that $t\in E_q$ implies that
\begin{equation}
\left(\sum_{q\in A_2}\left(\sum_{t\in E_q}|\la_t|^{p_{\zeta}}\right)^{p_\eta/p_{\zeta}}\right)^{1/p_\eta} \leqslant 2^{1/p_{\zeta}}\left(\sum_{t=1}^m|\lambda_t|^{p_{\zeta}}\right)^{1/p_{\zeta}}.\label{upper estimate 2}
\end{equation}
Combining relations \eqref{upper estimate 1} and \eqref{upper estimate 2} with $0<\theta\leqslant 1/4$, we get that $|f(\sum_{t=1}^m\lambda_t x_{j_{t}})|$ is bounded by the desired value. Note that for convenience we have implicitly assumed that $p_\eta<\infty$, but the case that $p_\eta=\infty$ would only require trivial modification.

\noindent {\em Case 2:} The functional $f$ is of order-$k^{\prime\prime}$ with $1\leqslant k^{\prime\prime} \leqslant k^\prime$.

Set 
\begin{eqnarray*}
t_0 &=& \min\{t: \ran f\cap \ran x_{j_t}\neq\varnothing\}\;\text{and}\\
q_0 &=& \min\{q: \max \supp f_q \geqslant \min\supp x_{j_{t_0 + 1}}\}.
\end{eqnarray*}
We shall prove the following:
\begin{equation}
\theta \left|\sum_{q>q_0}c_qf_q\left(\sum_{t=1}^m\la_tx_{j_t}\right)\right| < \theta \max\{|\lambda_t|:\;t>t_0\}.\label{action on tail}
\end{equation}
Since $(f_q)_{q=1}^d$ is admissible,
we have that $d \leqslant\max\supp x_{j_{t_0}}$. Also, $(f_q)_{q=1}^d$ is very fast growing and hence for $q >
q_0$ we have that
$$s(f_q) \geqslant \max \supp f_{q_0} \geqslant \min \supp x_{j_{t_0 +
1}}.$$
Moreover the functionals $f_q$ are of order strictly smaller than $k^\prime$,
therefore for $q>q_0$  and $t>t_0$, \eqref{constraint action} yields that
$
|f_q(x_{j_t})| < 1/\left(j_{t_0+1}\max\supp x_{j_{t_0}}\right)
$
and since $d\leqslant\max\supp x_{j_{t_0}}$, by keeping $t$ fixed, we obtain $\sum_{q>q_0}|c_qf_q(x_{j_t})| < 1/j_{t_0+1}$. Similarly, summing over the $t$ which are strictly greater than $t_0$, since $m \leqslant j_{t_0+1}$ we obtain:
\begin{eqnarray*}
\left|\sum_{q>q_0}c_qf_q\left(\sum_{t=1}^m\la_tx_{j_t}\right)\right| &=& \left|\sum_{q>q_0}c_qf_q\left(\sum_{t>t_0}^m\la_tx_{j_t}\right)\right|\\
&\leqslant& \sum_{t> t_0}|\la_t|\sum_{q>q_0}\left|c_qf_q\left(x_{j_t}\right)\right| \\
&<& \max_{t>t_0}|\lambda_t| (m/j_{t_0+1}) \leqslant \max_{t>t_0}|\lambda_t|.
\end{eqnarray*}
Thus, \eqref{action on tail} holds. We now observe the following:
\begin{equation}
\theta \left|\sum_{q<q_0}c_qf_q\left(\sum_{t=1}^m\la_tx_{j_t}\right)\right| = \theta \left|\sum_{q<q_0}c_qf_q\left(\la_{t_0}x_{j_{t_0}}\right)\right| \leqslant |\la_{t_0}|.\label{action on first}
\end{equation}
Moreover, the inductive assumption yields that
\begin{equation}
\left|f_{q_0}\left(\sum_{t=1}^m\la_tx_{j_t}\right)\right| \leqslant 2\left\|\left(\lambda_t\right)_{t=1}^m\right\|_{\ell_{p_{\zeta}}}.\label{action on middle}
\end{equation}
Combining \eqref{action on tail}, \eqref{action on first} and \eqref{action on middle} with $0<\theta\leqslant 1/4$ we get that $|f(\sum_{t=1}^m\lambda_t x_{j_{t}})|$ is bounded by the desired value.
 
The proof for the case in which $\zeta$ is finite is complete. Assume now that $\zeta = \xi_0 = \omega$ and pass to a 
subsequence of $(x_j)_j$ generating as a spreading model some sequence 
$(z_j)_j$. The previous case implies that $(z_j)_j$ is 2-dominated by  the unit vector basis of $\ell_{p_k}$ for all $k < \xi_0$ 
and hence, by taking a limit, it is also 2-dominated by the unit vector basis of  $\ell_{p_{\xi_0}}$ which yields the desired result.
\end{proof}

The next result explains that the $\al_{<k}$-indices of a given block sequence determine the spreading models admitted by it.

\begin{prop}\label{sm versus index}
Let $(x_j)_j$ be a normalized block sequence in $X$. Then $(x_j)_j$ admits an $\ell_{p_\zeta}$ spreading model, for some $\zeta\in[1,\xi_0]$. The following describes more precisely the spreading models admitted by $(x_j)_j$.

\begin{itemize}

\item [a.]Let $2\leqslant k < \xi_0$, then the following assertions are equivalent:
\begin{enumerate}

\item[(i)]   $\al_{<k}\{(x_j)_j\} > 0$ and $\al_{<k^\prime}\{(x_j)_j\} = 0$ for $1\leqslant k^\prime < k$.

\item[(ii)] There exists a subsequence of $(x_j)_j$ that generates an $\ell_{p_{k}}$ spreading model, while no subsequence of $(x_j)_j$ generates an $\ell_{p_{k^\prime}}$ spreading model for $1\leqslant k^\prime < k$.

\end{enumerate}
\item[b.] The following are equivalent
\begin{enumerate}

\item[(i)] $\al_{<1}\{(x_j)_j\} > 0$.

\item[(ii)] There exists a subsequence of $(x_j)_j$ that generates an $\ell_{p_{1}}$ spreading model.

\end{enumerate}

\item[c.] The following are also equivalent:
\begin{enumerate}

\item[(i)] For every $1\leqslant k < \xi_0$ we have that $\al_{<k}\{(x_j)_j\} = 0$.

\item[(ii)] Every subsequence of $(x_j)_j$ has a further subsequence generating an $\ell_{p_{\xi_0}}$ spreading model.

\end{enumerate}
\end{itemize}
Note that in the case $\xi_0$ is finite and $\xi_0 = k_0 + 1$, then c.(i) is equivalent to $\al_{<k_0}\{(x_j)_j\} = 0$.
\end{prop}

\begin{proof}
We shall only prove a. as the others are proved similarly, using Proposition \ref{general estimate} and Lemmas \ref{lower estimate}, \ref{upper estimate}. Assume that the first assertion of a. holds. Note that on every subsequence of $(x_j)_j$ the $\al_{<k-1}$-index is zero, and hence, applying Lemma \ref{upper estimate}, it has a further subsequence which admits a spreading model dominated by the unit vector basis of $\ell_{p_{k}}$. This in particular implies that no subsequence of $(x_i)_i$ generates an $\ell_{p_{k^\prime}}$ spreading model for $1\leqslant k^\prime < k$. Moreover, applying Lemma \ref{lower estimate} we pass to a subsequence $(x_{j_i})_i$, of $(x_j)_j$, generating some spreading model  dominating the usual vector basis of $\ell_{p_{k}}$. Since $\al_{<k-1}\{(x_j)_j\} = 0$, Lemma \ref{upper estimate} implies that this spreading model has to be  $\ell_{p_{k}}$. 

We assume now that the second assertion of a. holds. We first note that $\al_{<k}\{(x_j)_j\} > 0$. If this were not the case, then on every subsequence of $(x_j)_j$ the $\al_{<k}$-index would be zero and hence, by Lemma \ref{upper estimate}, every spreading model admitted by it is dominated by the unit vector basis of $\ell_{p_{k+1}}$. This means that no subsequence of $(x_j)_j$ can generate an $\ell_{p_{k}}$ spreading model, which is absurd.  Therefore the natural number $k_0 = \min\{k\in[1,\xi_0): \;\al_{<k}\{(x_j)_j\}>0\}$ is well defined and $k_0\leqslant k$. We shall prove that $k_0 = k$ and this will complete the proof.

Assume that $k_0 < k$ and apply Lemma \ref{lower estimate} to pass to a subsequence $(x_{j_i})_i$ of $(x_j)_j$ generating some spreading model which dominates the usual basis of $\ell_{p_{k_0}}$. If $k_0=1$ then by Proposition \ref{general estimate} we conclude that  $(x_{j_i})_i$ generates an $\ell_{p_1}$ spreading model, where $1 = k_0 < k$, which is absurd. If $1<k_0$, then $\al_{<k_0 - 1}\{(x_{j_i})_i\} = 0$ by and Lemma \ref{upper estimate} we conclude that $(x_{j_i})_i$ generates an $\ell_{p_{k_0}}$ spreading model, which is absurd for the same reasons.

\end{proof}

\begin{rem}
It is not hard to check that for every $1\leqslant k < \xi_0$, the $\al_{<k}$-index of the basis $(e_i)_i$ is zero and hence it only admits $\ell_{p_{\xi_0}}$ as a spreading model.
\end{rem}

\section{Spreading models of infinite dimensional subspaces of $X$.}
In the previous section we showed that every spreading model admitted by $X$ must be $\ell_p$ for some $p\in F$.
In this section we  show that, starting with a block sequence generating some spreading model, one may pass to a block sequence of it generating an other spreading model. We conclude that, in the case in which $F$ is finite, the spreading models admitted by every infinite dimensional subspace of $X$ are exactly the $\ell_{p}$ for $p\in F$. In the case which $F$ consists of an increasing sequence and its limit $p_{\xi_0}$, the spreading models admitted by every infinite dimensional subspace of $X$ may either be the $\ell_{p}$ for $p\in F$, or the $\ell_{p}$ for $p\in F\setminus\{p_{\xi_0}\}$. We start with two lemmas that describe the kind of block vectors one has to consider when switching from one spreading model to an other.

\begin{lem}\label{switch sm1}
Let $k$ be in $[1,\xi_0)$, $x_1 < \cdots < x_K$ be a finite normalized block sequence in $X$ that is  3-dominated by  the unit vector basis of $\ell_{p_k}^K$,  and set $x = K^{-1/p_k}\sum_{j=1}^K x_j$. If $f$ is a functional of order-$0$ in $W$ with $s(f) = m$ then the following holds:
\begin{equation}
|f(x)| < \frac{K^{1/p_{\xi_0}}}{K^{1/p_k}} + 2\frac{m^{1/p_{\xi_0}}}{m^{1/p_k}},
\end{equation}
where in the case $p_{\xi_0} = \infty$ we set $1/p_{\xi_0} = 0$.
\end{lem}

\begin{proof}
Let $f = \theta\sum_{q=1}^d(1/m)^{1/p_{\xi_0}^\prime}f_q$ be a functional of order-$0$ with $s(f) = m$ (recall that $d\leqslant m$). For convenience, we assume that $p_{\xi_0}<\infty$, and the proof for the case $\xi_0=\infty$ requires only trivial modifications.  By Lemma \ref{action analysis}, following the notation used there, applying H\"older's inequality for the pair $(p_{\xi_0},p_{\xi'_0})$, we obtain the following:
\begin{eqnarray}
\left|f(x)\right| &\leqslant& K^{-1/p_k}\left(\left(\sum_{j\in B}|g_t(x_{j})|^{p_{\xi_0}}\right)^{1/p_{\xi_0}} \right.\nonumber\\
 & &\left.+\theta\left(\sum_{q\in A_2}(1/m)^{1/p^\prime_{\xi_0}}\left|f_q\left(\sum_{j\in E_q}x_{j}\right)\right|\right)\right)\nonumber\\
&\leqslant& K^{-1/p_k}\left(K^{1/p_{\xi_0}} + 3\theta(1/m)^{1/p^\prime_{\xi_0}}\left(\sum_{q\in A_2}(\#E_q)^{1/p_k}\right)\right).\label{getting bored}
\end{eqnarray}
Recall that $\#A_2 \leqslant d \leqslant m$ and  
the last part of Lemma \ref{some details} gives that $\sum_{q\in A_2} \#E_q< 2K$. Combing these two facts gives us
\begin{eqnarray}
(1/m)^{1/p^\prime_{\xi_0}}\left(\sum_{q\in A_2}(\#E_q)^{1/p_k}\right) &\leqslant& (1/m)^{1/p^\prime_{\xi_0}}
m^{1/p^\prime_{k}}\left(\sum_{q\in A_2} \#E_q\right) ^{1/p_k}\nonumber\\
&<& \frac{m^{1/p^\prime_k}}{m^{1/p^\prime_{\xi_0}}}2^{1/p_k}K^{1/p_k}\nonumber\\
 &=& 2^{1/p_k}\frac{m^{1/p_{\xi_0}}}{m^{1/p_k}} K^{1/p_k}.\label{as if anybody would read this}
\end{eqnarray}

By combining relations \eqref{getting bored} and \eqref{as if anybody would read this} we achieve the desired upper bound.
\end{proof}

\begin{lem}\label{switch sm2}
Let $(x_j)_j$ be a normalized block sequence in $X$ and $2\leqslant k + 1 < \xi_0$ with $\al_{<k}\{(x_j)_j\} = 0$. Then there exists a subsequence $(x_{j_i})_i$ of $(x_j)_j$ such that for every $K\leqslant j_{i_1} < \cdots < j_{i_K}$  and every $f\in W$ of order at most $k$ with $s(f) = m$, we have that if $x = K^{-1/p_{k+1}}\sum_{t=1}^Kx_{j_{i_t}}$ then 
\begin{equation}
|f(x)| <  \frac{3 + K^{1/p_{\xi_0}}}{K^{1/p_{k+1}}} + 2\frac{m^{1/p_{\xi_0}}}{m^{1/p_{k+1}}},
\end{equation}
where in the case $p_{\xi_0} = \infty$ we set $1/p_{\xi_0} = 0$.
\end{lem}

\begin{proof}
By Lemma \ref{upper estimate} we may assume 
for every $K\leqslant j_1 < \cdots < j_K$ that $(x_{j_i})_{i=1}^K$ is  3-dominated by  the unit vector basis of $\ell_{p_{k+1}}^K$.
Using Proposition \ref{index char} we pass to a subsequence, again denoted by $(x_j)_j$ such that for any
$j \geqslant j_0\geqslant 2$, for any
$f\in W$ of order strictly smaller than $k$ with
$s(f) \geqslant \min\supp x_{j_0}$ we
have that
\begin{equation}
|f(x_j)| < \left(j_0\max\supp x_{j_0-1}\right)^{-1} .\label{constraint action blah blah}
\end{equation}
Let $K\leqslant j_1 < \cdots < j_K$, $x = K^{-1/p_{k+1}}\sum_{i=1}^Kx_{j_i}$ and $f = \theta\sum_{q=1}^dc_qf_q$ be a functional of order at most $k$.
Using a finite induction on $0\leqslant k^\prime\leqslant k$, we shall prove that for every $f = \theta\sum_{q=1}^dc_qf_q$ of order at most $k^\prime$ there is $i\inn$ such that the following holds:
\begin{equation}
|f(x)| < \left(\frac{1-\theta^i}{1-\theta}\right)\frac{2}{K^{1/p_{k+1}}}  + \frac{K^{1/p_0}}{K^{1/p_{k+1}}} + 2\frac{m^{1/p_{\xi_0}}}{m^{1/p_{k+1}}}.\label{I just had an omelette}
\end{equation}
The above in conjunction with  $0<\theta\leqslant 1/4$ clearly implies the desired result.

If a functional $f$ is of order-$0$. Then, by Lemma \ref{switch sm1} for $i=1$ we have that \eqref{I just had an omelette} holds. Assume that $f = \theta\sum_{q=1}^dc_qf_q$ is of order-$k^\prime$ with $0 < k^\prime\leqslant k$ and that \eqref{I just had an omelette} holds for every functional with order strictly smaller than $k^\prime$. Set 
\begin{eqnarray*}
t_0 &=& \min\{t: \ran f\cap \ran x_{j_t}\neq\varnothing\}\;\text{and}\\
q_0 &=& \min\{q: \max \supp f_q \geqslant \min\supp x_{j_{t_0 + 1}}\}.
\end{eqnarray*}
The same argument used to obtain \eqref{action on tail} and \eqref{action on first} in the proof of Lemma \ref{upper estimate} gives us the following:
\begin{equation}
\left|\theta\sum_{q\neq q_0}c_qf_q(x)\right| < 2/K^{1/p_{k+1}}.\label{not in middle}
\end{equation}
By the inductive assumption there exists $i\inn$ such that:
\begin{equation}
\theta|f_{q_0}(x)| <  \theta\left(\left(\frac{1-\theta^i}{1-\theta}\right)\frac{2}{K^{1/p_{k+1}}}  + \frac{K^{1/p_0}}{K^{1/p_{k+1}}} + 2\frac{s(f_{q_0})^{1/p_{\xi_0}}}{s(f_{q_0})^{1/p_{k+1}}}\right).\label{but now I would like some corn flakes (without sugar please)}
\end{equation}
By the definition of size for functionals which are not of order-$0$ we have that $s(f_{q_0}) \geqslant s(f)$ and hence combining \eqref{not in middle} and \eqref{but now I would like some corn flakes (without sugar please)} we conclude that:
\begin{equation*}
|f(x)| < \left(\frac{1-\theta^{i+1}}{1-\theta}\right)\frac{2}{K^{1/p_{k+1}}}  + \theta\left(\frac{K^{1/p_0}}{K^{1/p_{k+1}}} + 2\frac{s(f)^{1/p_{\xi_0}}}{s(f)^{1/p_{k+1}}}\right).
\end{equation*}
\end{proof}

 The next proposition allows us to pass from a block sequence admitting an $\ell_{p_{\xi_0}}$ spreading model to a further block admitting $\ell_{p_1}$ spreading model and from block sequence admitting an $\ell_{p_{k}}$ spreading model to a further block admitting an $\ell_{p_{k+1}}$ spreading model. In the case that $\xi_0 <\omega$, we use this to show that the spreading models in every subspace are exactly $\ell_p$ for $p \in \{p_1,p_2,\cdots,p_{\xi_0-1},p_{\xi_0}\}$. In the case that $\xi_0=\omega$ we require an additional argument to show that we have $\ell_{p_k}$ spreading model for every $k<\omega$ since we are not able to show that every block subspace admits an $\ell_{p_{\xi_0}}$ spreading model.

\begin{prop}\label{block to change}
Let $(x_j)_j$ be a normalized block sequence in $X$.
\begin{enumerate}

\item[(i)] If $(x_j)_j$ generates an $\ell_{p_{\xi_0}}$ spreading model, then there exists a further normalized block sequence $(y_j)_j$ of $(x_j)_j$ that generates an $\ell_{p_1}$ spreading model.

\item[(ii)] If $1 \leqslant k <  \xi_0$ and $(x_j)_j$ generates an $\ell_{p_k}$ spreading model, then there exists a further normalized block sequence $(y_j)_j$ of $(x_j)_j$ that generates an $\ell_{p_{k+1}}$ spreading model.

\end{enumerate}
\end{prop}

\begin{proof}
Let $(x_j)_j$ be a normalized block sequence in $X$, generating an $\ell_{p_\zeta}$ spreading model, for some $1\leqslant\zeta\leqslant \xi_0$. Note that by Proposition \ref{general estimate} and Lemma \ref{upper estimate} we may assume that for every $K \leqslant j_1 < \cdots < j_K$ we have that $\|x_{j_1} + \cdots + x_{j_K}\| \leqslant 3\cdot K^{1/p_\zeta}$. We distinguish three cases concerning $\zeta$, namely $\zeta=\xi_0, \zeta = 1$, and $2 \leqslant \zeta < \xi_0$. We shall only consider the first two cases, as the last one is proved in an identical manner as the case $\zeta=1$ and uses Lemma \ref{switch sm2} instead of Lemma \ref{switch sm1}.

\noindent{\em Case 1:} $\zeta = \xi_0$. For every $j\inn$ choose $f_j\in W$ with $f_j(x_j) = 1$ and $\ran f_j \subset \ran x_j$. Choose an increasing sequence of finite subsets of the natural numbers $(E_j)_j$ with $\#E_j \leqslant \min E_j$ and $\lim_j\#E_j = \infty$.  For $j\inn$ define $y_j^\prime = (\#E_j)^{-1/p_{\xi_0}}\sum_{i\in E_j}x_i$, $y_j = \|y_j^\prime\|^{-1}y_j^\prime$ and $g_j = \theta\sum_{i\in E_j}(\#E_j)^{-1/p_{\xi_0}^\prime}f_i$. Then we have the following:
\begin{enumerate}

\item[(a)] The sequence $(y_j)_j$ is a normalized block sequence of $(x_j)_j$ and for every $j\inn$ we have that $g_j(y_j) \geqslant \theta/3$.

\item[(b)] The functional $g_j$ is of order-$0$ with $s(g_j) = \#E_j$ for all $j\inn$.

\end{enumerate}
Note that $\lim_js(g_j) = \infty$ and therefore, passing to a subsequence, we may assume that $(g_j)_j$ is a very fast growing sequence of functionals of order-$0$. We conclude that $\al_{<1}\{(x_j)\} > 0$ and by Proposition \ref{sm versus index} we have that $(y_j)_j$ is the desired sequence.

\noindent {\em Case 2:} $\zeta = 1$. By Proposition \ref{sm versus index} we have that $\al_{<1}\{(x_j)_j\} > 0$ and hence, by passing to a subsequence, there exists $\ee>0$ and a very fast growing  sequence $(f_j)_j$ of order-$0$ functionals such that $\ran f_j\subset \ran x_j$ and $f_j(x_j) > \ee$ for all $j\inn$. Choose an increasing sequence of finite subsets of the natural numbers $(E_j)_j$ with $\min E_j \leqslant \#E_j$ and $\lim_j\#E_j = \infty$.  For $j\inn$ define $y_j^\prime = (\#E_j)^{-1/p_1}\sum_{i\in E_j}x_i$, $y_j = \|y_j^\prime\|^{-1}y_j^\prime$ and $g_j = \theta\sum_{i\in E_j}(\#E_j)^{-1/p_1^\prime}f_j$ (if $p_1 = 1$ take $g_j = \theta\sum_{i\in E_j}f_j$ instead). Then we have the following:
\begin{enumerate}

\item[(a')] The sequence $(y_j)_j$ is a normalized block sequence of $(x_j)_j$ and for every $j\inn$ we have that $g_j(y_j) \geqslant \ee\theta/3$.

\item[(b')] The functional $g_j$ is of order-$1$ with $s(g_j) \geqslant \max\{s(f_i): i\in E_j\}$ for all $j\inn$.

\end{enumerate}
Once more, $\lim_js(g_j) = \infty$ and as before we conclude that $\al_{<2}\{(x_j)\} > 0$. By Proposition \ref{sm versus index} it remains to observe that $\al_{<1}\{(y_j)\}_j = 0$, which is an easy consequence of the definition of the $y_j$'s and Lemma \ref{switch sm1}.
\end{proof}

\begin{rem}
The proof of Proposition \ref{block to change} implies that the space $X$ does not admit an $\ell_{p_\zeta}^{\mathcal{S}_2}$ spreading model for any $1\leqslant\zeta\leqslant \xi_0$. For the definition of an $\ell_p^{\mathcal{S}_k}$ spreading model see \cite[Definition 1.1]{ABM}.
\end{rem}

\begin{cor}
The space $X$ is reflexive.
\end{cor}

\begin{proof}
Proposition \ref{block to change} implies that neither $c_0$ nor $\ell_1$ embed into $X$. By James' well known theorem for spaces with an unconditional basis we conclude that $X$ is reflexive.
\end{proof}

\begin{rem}
If $(z_j)_j$ is a spreading model generated by a non-norm convergent (not necesssarily Schauder basic) sequence in $X$, then \cite[Remark 5, page 581]{AKT} the reflexivity of the space and Proposition \ref{sm versus index} imply that, although the sequence $(z_j)_j$ need not be a Schauder basis for $Z = \overline{\langle\{z_j:\;j\inn\}\rangle}$, the space $Z$ must be isomorphic to $\ell_p$, for some $p\in F$.
\end{rem}

\begin{lem}\label{how to skip the smallest spreading model}
Let $1\leqslant k < \xi_0,$ $K\inn$ and $(x_j)_j$ be a sequence in $X$ generating an $\ell_{p_k}$ spreading model. Then for every $j_0\inn$ there exists a normalized vector $x\in\text{span}(x_j)_{j\geqslant j_0}$ and a functional $f$ of order-$0$ with $s(f) = K$ such that
\begin{equation}
f(x) \geqslant \frac{\theta}{3} K^{1/p_{\xi_0} - 1/p_k},
\end{equation}
where in the case $p_{\xi_0} = \infty$ we set $1/p_{\xi_0} = 0$.
\end{lem}

\begin{proof}
We may clearly assume that $(x_j)_j$ is normalized. Proposition \ref{sm versus index} in conjunction with \ref{upper estimate} imply that we may choose $j_0 \leqslant j_1 < \cdots <j_K$ such that if $y = K^{-1/p_k}\sum_{i=1}^Kx_{j_i}$ then $\|y\| \leqslant 3$. Choose $f_1,\ldots, f_K$ with $\ran f_i\subset \ran x_{j_i}$ and $f_i(x_{j_i}) = 1$ for $i=1,\ldots,K$ and define $f = \theta\sum_{i=1}^K(1/K)^{-1/p_{\xi_0}^\prime}f_i$ and  $x = y/\|y\|$. 
\end{proof}

\begin{thm}\label{first part main theorem}
Let $F =\{p_\zeta:\;1\leqslant \zeta\leqslant \xi_0\}$ and let $Y$ be an infinite dimensional subspace of $X$. Then there exists a dense subset  $G$ of $F$ such that the spreading models admitted by $Y$ are exactly the $\ell_{p}$, for $p\in G$. In particular, $\ell_{p}$ is finitely block represented in every block subspace of $X$ for every $p\in F$.
\end{thm}

\begin{proof}
Let $Y$ be a block subspace of $X$. We observe the following:
\begin{enumerate}

\item[(i)] Every spreading model admitted by $Y$ is equivalent to the unit vector basis of $\ell_{p_\zeta}$ for some $1\leqslant\zeta\leqslant \xi_0$. In particular, there exists $1\leqslant \zeta_0 \leqslant \xi_0$ such that $Y$ admits an $\ell_{p_{\zeta_0}}$ spreading model.

\item[(ii)] If $1\leqslant k <\xi_0$ and $Y$ admits an $\ell_{p_k}$ spreading model and, then $Y$ also admits an $\ell_{p_{k+1}}$ spreading model.

\item[(iii)] If $Y$ admits an $\ell_{p_{\xi_0}}$ spreading model then $Y$ also admits an $\ell_{p_1}$ spreading model.

\item[(iv)] There exists $1\leqslant k_0 <\xi_0$ such that $Y$ admits an $\ell_{p_{k_0}}$ spreading model.

\end{enumerate}
The statement (i) follows from Proposition \ref{sm versus index}. Statements (ii) and (iii) follow  from Proposition \ref{block to change}, while (iv) follows from the first and the third. We now distinguish two cases regarding whether $\xi_0$ is finite or not.

{\em Case 1:} If $\xi_0$ is finite, statement (i), statement (ii), and a finite inductive argument yield that $Y$ admits an $\ell_{p_{\xi_0}}$ spreading model. By (iii) we have that $Y$ admits an $\ell_{p_1}$ spreading model. Once more, by a finite induction we obtain that $G = F$ is the desired set.

{\em Case 2:} If $\xi_0 = \omega$ we shall prove that for every $1\leqslant k <\xi_0$, $Y$ admits an $\ell_{p_k}$ spreading model. This in particular implies that $G = F$ or $G=F\setminus\{p_{\xi_0}\}$ is the desired set. By (ii) it is sufficient to show that $Y$ admits an $\ell_{p_1}$ spreading model. By Proposition \ref{sm versus index} it is enough to find a normalized block sequence $(x_j)_j$ in $Y$ and  a very fast growing sequence of functionals $(f_j)_j$ of order-$0$ with $f_j(x_j) > \theta/4$, i.e. $\al_{<1}\{(x_j)_j\} > 0$.

Choose a normalized vector $x_1$ in $Y$ and a functional $f\in W$ with $f(x_1) = 1$ and set $f_1 = \theta f$. Then $f_1$ is of order-$0$ with $s(f) = 1 $ and $f_1(x_1) > \theta/4$. Assume that we have chosen normalized vectors $x_1 < \cdots < x_j$ and a very fast growing sequence  of functionals $f_1,\ldots,f_j$ of order-$0$ with $f_i(x_i) > \theta/4$ for $i=1,\ldots,j$. By (iv), there exists $1\leqslant k_0<\xi_0$ such that $Y$ admits an $\ell_{p_{k_0}}$ spreading model.
Fix $K > \max\supp f_j$ and choose $k_0\leqslant k <\xi_0$  such that $K^{1/p_{\xi_0} - 1/p_k} > 3/4$ (recall that $\lim_kp_k = p_{\xi_0}$). By (ii) we may choose a sequence $(y_i)_i$ in $Y$ generating an $\ell_{p_k}$ spreading model. Choose $i_0\inn$ with $\min\supp y_{i_0} \geqslant (\max\supp x_j)^2$ and apply Lemma \ref{how to skip the smallest spreading model} to find the desired pair $x_{j+1}, f_{j+1}$.
\end{proof}

\begin{rem}
In the case that $F$ is finite, then clearly the spreading models admitted by every block subspace of $X$ are exactly the $\ell_p$, for $p\in F$. In the case that $F$ consists of an increasing sequence and its limit $p_{\xi_0}$, then it is easily checked that exactly one of the following holds:
\begin{itemize}

\item[(i)] The spreading models admitted by every block subspace of $X$ are exactly the $\ell_p$, for $p\in F$.

\item[(ii)] There exists  a block subspace $Y$ of $X$, such that the spreading models admitted by every further block subspace of $Y$ are exactly the $\ell_p$, for $p\in F\setminus\{p_{\xi_0}\}$.

\end{itemize}
We were unable to determine which one of the above holds, in either case  however on some subspace $Y$ of $X$, the set of spreading models admitted by every further subspace of $Y$ is stabilized.
\end{rem}

\section{The set of Krivine $p$'s of the space $X$.}
In this section we prove that for any $p\notin F = \{p_\zeta:\; 1\leqslant\zeta\leqslant \xi_0\}$,  $\ell_p$ is not finitely block represented in the space $X$. We conclude that the set of $p$'s that are finitely block represented in every block subspace of $X$ is exactly the set $F$, which is not connected.


We begin with the following Lemma, whose proof we omit as it follows from
the same argument as the proof of Lemma \ref{switch sm1}.

\begin{lem}\label{wrong order wont norm}
Let $p\in[p_1,p_{\xi_0}]\setminus F$. Suppose $\ee>0$ and $(x_j)_{j=1}^N$ is a finite block sequence in $X$ which is $(1+\ee)$-equivalent to the unit vector basis of $\ell_p^N$. If $1\leqslant \zeta \leqslant \xi_0$ is such that $p < p_\zeta$ and $f$ is a functional of order-$\zeta$, then we have the following estimate:
\begin{equation}
\left|f\left(\frac{1}{N^{1/p}}\sum_{j=1}^Nx_j\right)\right| < (1+\ee)\left(\frac{N^{1/p_{\zeta}}}{N^{1/p}} + 2\theta\right).
\end{equation}
\end{lem}

The next lemma follows directly from the above lemma.

\begin{lem}
Suppose that $(x_j)_{j=1}^N$ is a finite block sequence in $X$ that is $(1+\ee)$-equivalent to the unit vector basis of $\ell_p^N$, $1\leqslant k < \xi_0$ satisfies $p<p_{k+1}$ and $N$ satifies
$$N^{1/p_{k+1}-1/p}+2\theta < (1+\ee)^{-2}.$$
If $f \in W$ satisfies $f(N^{-1/p}\sum_{j=1}^N x_j) \geqslant 1/(1+\ee)$
then $f$ {has non-zero order less than or equal to $k$}.\label{large size is bad}
\end{lem}

We are now ready to prove the second main theorem.

\begin{thm}\label{second part main theorem}
For all $p\in [1,\infty]\setminus F$ there exists $K\in\N$ and $\vp>0$ such that no
block sequence $(x_j)_{j=1}^K$ in $X$ is $(1+\vp)$-equivalent to
the unit vector basis of $\ell_p^K$.
\end{thm}

\begin{proof}

Let $p\in [1,\infty]\setminus F$. If $p \notin [p_1,p_{\xi_0}]$, then the result clearly follows from Proposition \ref{general estimate}. Otherwise, we have that $p\in[p_1,p_{\xi_0}]\setminus F$. Find $k \inn$ so that $p_k< p < p_{k+1}$. Find $N,M \inn$ and $\ee>0$ as follows: 

Choose $N\in\N$ such that
\begin{eqnarray}
N^{1/p}&>&2+\theta(N-2)^{1/p}\;\text{and}\label{N fixed for good 1}\\
\frac{N^{1/p_{k + 1}}}{N^{1/p}} &<& 1 - 2\theta.\label{N fixed for good 2}
\end{eqnarray}
Now that $N$ is fixed,
we choose $\ee>0$ such that:
\begin{eqnarray}
\frac{1}{1+\vp}N^{1/p}&>&2+(1+\vp)\theta(N-2)^{1/p},\label{e fixed for good 1}\\
N^{1/p}&>&(1+\vp)^2 (N-1)^{1/p}\;\text{and}\label{e fixed for good 2}\\
\frac{N^{1/p_{k + 1}}}{N^{1/p}} &<& \frac{1}{(1+\ee)^2}-2\theta\label{e fixed for good 3}.
\end{eqnarray}
We set
\begin{equation}
\Theta =\frac{1}{1+\vp}N^{1/p}-(1+\vp) (N-1)^{1/p}.\label{this is Theta}
\end{equation}
Notice (\ref{e fixed for good 2}) implies that $\Theta >0$. 
Finally, let $M \inn$ so that
\begin{equation}
M^{1/p_{k}}\Theta>(1+\vp)M^{1/p}. \label{here is M}
\end{equation}

Let $K=(N-1)M+1$ and consider the following normalized block sequence which{, towards a contradiction,} we assume is $(1+\ee)$-equivalent to the unit vector basis of $\ell_p^{K}$  and that $M<\min \supp x_1$.
$$x_1 < x_2^1 < x_3^1< \cdots < x_N^1< x_2^2 < x_3^2 < \cdots <x_{N}^{M-1}< x_2^M < \cdots < x_N^M.$$
(i.e. $x_j^m < x_i^m$ for $i<j$ and $x_N^m<x_2^{m+1}$). { Let us mark the following, which is obviously true.
\begin{itemize}
\item[(a)] For each $m$ with $1 \leqslant m \leqslant M$ the block sequence  $x_1 < x_2^m < \cdots < x_N^m$ is $(1+\ee)$-equivalent to the unit vector basis of $\ell_p^N$. 
\end{itemize}
}

Fix $m$ with $1 \leqslant m \leqslant M$. For notational reasons we set $x^m_1=x_1$. Find $g_m \in W$ with $g_m( \sum_{i=1}^N x_i^m) \geqslant \frac{N^{1/p}}{(1+\ee)}$. By Lemma \ref{large size is bad} and (\ref{e fixed for good 3}) we conclude that $g_m$ has non-zero order less than or equal to $k$. Let
$$g_m = \theta \sum_{q=1}^{d_m} c_{m,q}f_{m,q}.$$
be the functionals decomposition according to Remark \ref{useful for inductive proofs on the norming set and other such stuff}{, i.e.:
\begin{itemize}
\item[(b)] the coefficients $(c_{m,q})_{q=1}^{d_m}$ are in the unit ball of $\ell_{p_k^\prime}$ for $m=1,\ldots,M$ and

\item[(c)] the sequence $(f_{m,q})_{q=1}^{d_m}$ is an admissible and very fast growing sequence of functionals, each one of which has order strictly smaller than $k$.
\end{itemize}
}
Define 
$$q_m= \min \{q :   \min \supp x_2^m\leqslant \max\supp f_{m,q}\}$$
We will prove the following three claims:
\begin{itemize}
\item[(i)] $g_m(x_1) >\Theta$, $g_m(x_N^m) >\Theta$ and $d_m \leqslant \max\supp x_1$.
\item[(ii)] The number $q_m$ exists, $\max\supp f_{m,q_m} < \min \supp x_{N}^m$ and $q_m < d_m$.
\item[(iii)] $(\min\supp x_2^m)^2 < \min \supp f_{m,q_m+1}$ and $\min\supp x_2^m < s(f_{m,q_m+1})$.
\end{itemize}

Item (i): Using  { that $g_m( \sum_{i=1}^N x_i^m) \geqslant \frac{N^{1/p}}{(1+\ee)}$, (a)} and (\ref{this is 
Theta}) we have
\begin{equation}
\begin{split}
g_m(x_1)& =g_m\left(\sum_{j=1}^N x_j^m\right) - g_m\left(\sum_{j=2}^N x_j^m\right)\\
&  \geqslant \frac{1}{1+\vp}N^{1/p}-(1+\vp) (N-1)^{1/p} = \Theta >0.
\end{split}
\end{equation}
The same argument works to show that $g_m(x_N^m) >\Theta$. If $d_m > \max\supp x_1$ then $\max \supp x_1 < \min \supp g_m$ which implies that $g_m(x_1)=0$. This contradiction tells us that $d_m \leqslant \max\supp x_1$.

Item (ii): If $q_m$ did not exist then $\min\supp x_2^m> \max\supp f_{m,q}$ for all $q$ and so $g_m(\sum_{j=1}^N x_j^m)= g_m(x_1) \leqslant 1$. On the other hand we clearly have $g_m(\sum_{j=1}^N x_j^m)>2$ and so $q_m$ exists. 

If $\max\supp f_{m,q_m} \geqslant \min \supp x_{N}^m$ then we have:
\begin{equation}
\begin{split}
g_m\left(\sum_{j=1}^N x_j^m\right)& =g_m(x_1) + \theta c_{m,q_m} f_{m,q_m}\left(\sum_{j=2}^{N-1} x_j^m\right)+ g_m(x_N^m)\\
& \leqslant 2 + (1+\vp)\theta(N-2)^{1/p} < \frac{N^{1/p}}{(1+\ee)}.
\end{split}
\end{equation}
The last inequality uses (\ref{e fixed for good 1}). This contradicts that fact that $g_m(\sum_{j=1}^N x_j^m)\geqslant \frac{N^{1/p}}{(1+\ee)}$. 

Using item (i) we have $g_m(x_N^m) > \Theta$. This fact combined with the fact that $\max\supp f_{m,q_m} < \min \supp x_{N}^m$ gives us $q_m<d_m$. 

Item (iii): By definition of $q_m$ and the fact that $(f_{m,q})_{q=1}^{d_m}$ is very fast growing  
$$(\min \supp x_2^m)^2\leqslant (\max\supp f_{m,q_m})^2 < \min \supp f_{m,q_m+1} \mbox{ and } $$
$$\min \supp x_2^m < s(f_{m,q_m+1}). $$ 
This proves (iii).

Note that $q_m+1 \leqslant d_m$ by item (ii). Define
$$f_m := \restr{g_m}{\ran x_{N}^m} =\theta\sum_{q=q_m +1}^{d_m} c_{m,q}\restr{f_{m,q}}{\ran x_{N}^m}.$$
We claim that 
\begin{equation}
\frac{1}{M^{1/p'_k}} \sum_{m=1}^M f_m \in W.
\label{it is in W}
\end{equation}
We first assume that (\ref{it is in W}) holds and finish the proof of our theorem as follows:
$$M^{1/p}(1+\ee)\geqslant \frac{1}{M^{1/p'_k}} \sum_{m=1}^M f_m \left(\sum_{m=1}^M x_N^m\right) = \frac{1}{M^{1/p'_k}} \sum_{m=1}^M f_m (x_N^m) \geqslant M^{1/p_k} \Theta.$$
This contradicts (\ref{here is M}).

All that remains to prove is (\ref{it is in W}). Note that 
{$$\frac{1}{M^{1/p'_k}} \sum_{m=1}^M f_m =  \theta\sum_{m=1}^M \sum_{q=q_m +1}^{d_m} \left(c_{m,q}/M^{1/p'_k}\right)\restr{f_{m,q}}{\ran x_{N}^m}.$$
Using (b) we obtain that $\sum_{m=1}^M \sum_{q=q_m +1}^{d_m} (c_{m,q}/M^{1/p'_k})^{p_k'} \leqslant 1$ and therefore
i}t suffices to show that $((f_{m,q})_{q=q_m+1}^{d_m})_{1 \leqslant m \leqslant M}$ is { an} admissible and very fast growing { sequence of functionals, each one of which has order strictly smaller than $k$, which will imply that $f$ is a funtional in $W$ of order-$k$}. First we check admissibility:
\begin{equation}
\begin{split}
 \sum_{m=1}^M d_m & \leqslant M \max \supp x_1 \\
& \leqslant \min\supp x_2^1 \cdot \min \supp x_2^1 \\
& \leqslant \min\supp f_{1,q_1+1}.
\end{split}
\end{equation}
The first inequality follows from item (i), the second from that fact that $M< \max \supp x_1<\min\supp x_2^1$ and the third comes from item (iii) (for $m=1$).

{ Note that by (b) the functionals under consideration have order strictly smaller than $k$ and } for each $m$ with $1 \leqslant m \leqslant M$ the collection $(f_{m,q})_{q=q_m+1}^{d_m}$ is very fast growing{. At last,} it suffices to show for each $m\in\N$ with $2 \leqslant m \leqslant M$ that 
$$(\max\supp f_{m-1,d_{m-1}})^2< \min\supp f_{m,q_{m}+1} \mbox{ and }$$
$$\max\supp f_{m-1,d_{m-1}}<s(f_{m,q_{m}+1}).$$ 

This, however, follows from item (iii) since $$\max\supp f_{m-1,d_{m-1}}< \min\supp x_2^m.$$ This proves the claim and finishes the proof of the theorem.
\end{proof}

We are interested in three problems related to the present work.

\begin{problem}
Let $1<p_1<p_2< \infty$. Is the space $X_{p_1,p_2}$ constructed in this paper super-reflexive?
\end{problem}

\begin{problem}
Let $1 \leqslant p_1 < p_2 \leqslant \infty$. Does there exist a space so that in every block subspace the Krivine set is $[p_1,p_2]$? More generally, which types of closed sets can be hereditary Krivine sets?
\end{problem}

\begin{problem}
Let $F\subset[2,\infty)$ be finite. Does there exist a Banach space $X$ such that 
for every infinite dimensional subspace $Y$ of $X$, $\ell_p$ is finitely represented in $Y$ if and only if $p\in\{2\}\cup F$?  In particular, does our construction satisfy this?
\end{problem}

\end{document}